\documentclass[12pt,reqno]{amsart}
\usepackage{amsmath,amsfonts,amssymb,bbm}

\usepackage{graphicx}%%, caption}
\usepackage{pgf}
\usepackage{tikz}
\usepackage{pgfplots}\pgfplotsset{compat=1.9}
\usepgfplotslibrary{fillbetween}

\usepackage{enumerate}
\usepackage{comment}

\usepackage{caption}
\usepackage{subcaption}

\usepackage[colorlinks=true,linkcolor=blue,citecolor=magenta,urlcolor=cyan]{hyperref}

\newtheorem{theorem}{Theorem}[section]

\newtheorem{lemma}[theorem]{Lemma}

\newtheorem{remark}[theorem]{Remark}

\usepackage[utf8]{inputenc} % allow utf-8 input
\usepackage[T1]{fontenc}    % use 8-bit T1 fonts
\usepackage{lmodern}
\usepackage{hyperref}       % hyperlinks
\usepackage{url}            % simple URL typesetting
\usepackage{booktabs}       % professional-quality tables
\usepackage{nicefrac}       % compact symbols for 1/2, etc.
\usepackage{microtype}      % microtypography

\usepackage{color}

\def\R{\ensuremath{\mathbb{R}}}
\def\N{\ensuremath{\mathbf{N}}}

\def\sgn{\mathrm{sgn}}

\newcommand{\wf}{\widetilde{F}}

\renewcommand{\phi}{\varphi}
\newcommand{\eps}{\varepsilon}
\newcommand{\from}{\colon}

\newcounter{Ax}
\newcounter{Bx}
\newcounter{Cx}

\newcommand{\itemA}{%
    \addtocounter{Ax}{1}
    \item[A\theAx.]}

\newcommand{\itemC}{%
    \addtocounter{Cx}{1}
    \item[C\theCx.]}

\begin{document}
\title{
Interval maps mimicking circle rotations}
\begin{abstract}
We investigate the dynamics of maps of the real line whose behavior on an invariant interval is close to a rational rotation on the circle. We concentrate on a specific two-parameter family, describing the dynamics arising from models in game theory, mathematical biology and machine learning.
%%that serves as a model in problems in mathematical biology and ******************. 
If one parameter is a rational number, $k/n$, with $k,n$ coprime, and the second one is large enough, we prove that there is a periodic orbit of period $n$. It behaves like an orbit of the circle rotation by an angle $2\pi k/n$ and attracts trajectories of Lebesgue almost all starting points. We also discover numerically other interesting phenomena. While we do not give rigorous proofs for them, we provide convincing explanations.
\end{abstract}

\author
[J.~Bielawski]
{Jakub Bielawski$^1$}
\address{$^1$\,%
Department of Mathematics, Krakow University of Economics, Rakowicka 27, 31-510 Krak\'{o}w, Poland}
\email{bielawsj@uek.krakow.pl}

\author
[T.~Chotibut]
{Thiparat Chotibut$^2$}
\address{$^{2}$\,%
Chula Intelligent and Complex Systems Lab, Department of Physics, Chulalongkorn University, Bangkok 10330, Thailand}
\email{thiparat.c@chula.ac.th}

\author
[F.~Falniowski]
{Fryderyk Falniowski$^{3}$}
\address{$^{3}$\,%
Department of Mathematics, Krakow University of Economics, Rakowicka 27, 31-510 Krak\'{o}w, Poland}
\email{falniowf@uek.krakow.pl}

\author
[M.~Misiurewicz]
{Micha{\l} Misiurewicz$^{4}$}
\address{$^{4}$\,%
Department of Mathematical Sciences, Indiana University Indianapolis, 402 N. Blackford Street, Indianapolis, Indiana 46202, USA}
\email{mmisiure@iu.edu}

\author
[G.~Piliouras]
{Georgios Piliouras$^{5}$}
\address{$^{5}$\,%
Google DeepMind, London EC4A 3TW, United Kingdom}
\email{gpil@google.com}


\maketitle

\section{Introduction}

This paper is mainly focused on the study of a remarkable family of maps
$F\from \R\to\R$, given by
\begin{equation}\label{i1}
F(x)=x+b-\frac1{e^{-ax}+1},
\end{equation}
parametrized by two parameters $a>0$ and $b\in (0,1)$. This family,
with different parametrization,
arises from the update rule in mathematical biology (two predator-one-prey
model), see ~\cite{eirola1996chaotic,kryzhevich2021bistability}, as well as in  machine learning and game theory,
%%arises from problems in mathematical
%%biology (two predator-one-prey model),
%%see~\cite{eirola1996chaotic,kryzhevich2021bistability}, as well as in
%%the artificial intelligence problems,
see~\cite{bielawskiheterogeneity,bielawski2022memory,chotibut2021family}. Since apparently
the paper~\cite{eirola1996chaotic} by
Eirola, Osipov and S\"oderbacka is the first one where it appeared, we
will call maps of this family \emph{EOS maps}.

For important range of parameters, those maps have an invariant
interval that attracts all trajectories. And on this interval, if we
identify its endpoints (to get a circle) it looks almost like a circle
rotation. For a rational rotation on a circle, every point is periodic
of the same period. If our interval map is (mostly) very close to this
rotation, we may expect that at least one periodic orbit of this
period and with the same permutation will be present. We prove that
for appropriate range of parameters this is the case, and additionally
this surviving orbit is attracting.

From purely mathematical point of view, the problem we tackle is as
follows. We start with a rotation on the circle. Then we cut the
circle at one point, making it an interval. Our map develops a (big)
jump discontinuity. In order to make it continuous, we modify it in a
small neighborhood of the point of discontinuity. Will the resulting
continuous interval map still be in some sense similar to the circle
rotation? The answer is yes, if we do it the right way, although we
add a lot of dynamics we do not know much about.

We also investigate the family of EOS maps numerically and discover
some interesting phenomena. For example, if the rotation of the circle
is by $k/n$ (that is, by the angle $2\pi k/n$), the bifurcation diagrams
for the same $n$, but different $k$, look very similar to each other.
We explain this phenomenon, although we do not provide a rigorous
proof of it.

The paper is organized as follows. In Section~\ref{sec-naa} we
introduce notation and state conditions that guarantee the results we
want. In Section~\ref{Section3} we prove our main theorem
(Theorem~\ref{thmpern}). In Section~\ref{seceos} we show that for a
rational $b$ and sufficiently large $a$ the corresponding EOS map
satisfies the assumptions of this theorem. In Section~\ref{secbif} we
describe and explain other interesting properties of the family of the
EOS maps.

\section{Notation and assumptions}
\label{sec-naa}
Let the map $F\colon \mathbb{R}\to\mathbb{R}$ be given by
\begin{equation}\label{en1}
F(x)=x+b-g(x),
\end{equation}
where $b \in (0,1)$ is a parameter and $g\colon \mathbb{R}\to (0,1)$ is a $C^3$ map. We will assume that its Schwarzian derivative, that is
\[SF=\frac{F'''}{F'}-\frac 32 \left(\frac{F''}{F'}\right)\]
is negative.
We make the following assumptions on the map $g$:
\begin{enumerate}[(A)] 
\item \label{aa} there exist $y_-,y_+$, with $b-1<y_-<0<y_+<b$ such that
\begin{enumerate}
\itemA $0<g'(x)<1$ on $(-\infty,y_-)\cup (y_+,\infty)$,
\itemA $g'(y_-)=g'(y_+)=1$,
\itemA $g'(x)>1$ on $(y_-,y_+)$,
\end{enumerate}
\item \label{bb} $g(b-1)<b<g(b)$.
\end{enumerate}

In particular, $g$ is an increasing diffeomorphism.

Condition \eqref{aa} implies that $F$ is increasing on $(-\infty,y_-)$ and $(y_+,\infty)$, and
decreasing on $(y_-,y_+)$, and condition \eqref{bb} guarantees that
\begin{equation}
F(b-1)>b-1\ \ \textrm{and}\ \ F(b)<b.
\end{equation} 
Moreover, by the definition of $g$ we have $F(y_+)>y_++b-1>b-1$, and $F(y_-)<b$, so \[
F([b-1,b])\subset [b-1,b].
\]
Because $F(x)=x$ if and only if $g(x)=b$, thus as $F$ is increasing outside of $[b-1,b]$ and, by \eqref{bb}, $F(x)>x$ to the left of $b-1$, $F(x)<x$ to the right of $b$, the trajectory of every point outside of $[b-1,b]$ eventually falls into it. % and  the fixed point of $F$ belongs to $[b-1,b]]$, 
%the interval $[b-1,b]$ is attracting. 
Thus, conditions \eqref{aa} and \eqref{bb} guarantee that $[b-1,b]$ is a globally attracting invariant interval for $F$.

%{\color{blue} [FF: Rozdzieli{\l}em te wlasnosci, bo efektywnie dalej wykorzystujemy juz tylko wlasnosci (C). Przy okazji, (C3) i (C4) nie są potrzebne do pokazania niezmienniczosci [b-1,b] - wystarczy, ze $0<g(x)<1$ dla kazdego $x$.]}
Finally, we assume that 
\begin{itemize}
\item[(C)] \label{cc} there exists $\varepsilon>0$ such that
\begin{enumerate}
\itemC \label{c1} $g(x)<\varepsilon$ for any $x<-\frac{1}{2n}$,
\itemC \label{c2} $g(x)>1-\varepsilon$ for any $x>\frac{1}{2n}$,
\itemC $(n-1)\varepsilon<1-g(y_+)$ and
  $(n-1)\varepsilon<\frac{1}{2n}-\big(1-g(y_+)+y_+\big)$,
\itemC $(n-1)\varepsilon<g(y_-)$ and
  $(n-1)\varepsilon<\frac{1}{2n}-\big(g(y_-)-y_-\big)$.
\end{enumerate}
\end{itemize}
%{\color{blue} FF: Probowalem wygenerowac cross-referencing z C1 etc, ale na razie poleglem - powalcze z tym pozniej}

We denote the set of maps $F$ given by \eqref{en1} and satisfying all of the above assumptions %with $g$ fulfilling conditions (A-C) 
by $\mathfrak{F}$. We give an example of such map in Figure~\ref{f0}.

\begin{figure}[h!]
\begin{center}
\includegraphics[width=0.65\textwidth]{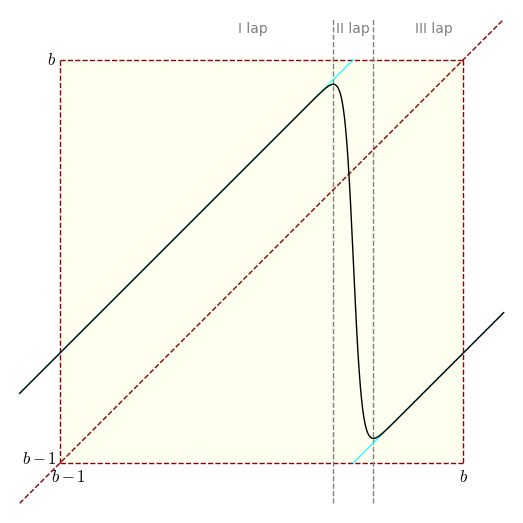}
\caption{The maps $F(x)=x+b-\frac{1}{e^{-90x}+1}$ (black) and $G$ (\textcolor{cyan}{cyan}) for $b=3/11$. The map $F$ belongs to $\mathfrak{F}$. The yellow region is the set $[b-1,b]^2$. % and
 % $a=110$. %Horizontal and vertical red lines are at the levels $b-1$, $b-2/3$, $b-1/3$ and $b$.}
  \label{f0}}
\end{center}
\end{figure}

\begin{remark}\label{r1-sym}
Observe that if we assume additionally symmetry \[g(-x)+g(x)=1,\] then 
\begin{itemize}
\item $y_-=-y_+$,
%\item in (A1) we can replace  $(-\infty,y_-)\cup (y_+,\infty)$ by
%  $(y_+,\infty)$,
%\item in (A3) we can replace $(y_-,y_+)$ by $[0,y_+)$,
\item (C1) and (C2) are equivalent,
\item (C3) and (C4) are equivalent.
\end{itemize}
\end{remark}

Let us also define a map $G\colon \R\to\R$ by
\begin{equation}\label{en2}
G(x)=\begin{cases}
x+b & \textrm{if}\ \ x<0,\\
x+b-1 & \textrm{if}\ \ x>0,
\end{cases}
\end{equation}
(we do not have to choose the value of $G$ at 0). In particular, we have
\begin{equation}\label{en3}
G(x)-F(x)=\begin{cases}
g(x) & \textrm{if}\ \ x<0,\\
g(x)-1 & \textrm{if}\ \ x>0.
\end{cases}
\end{equation}

\begin{remark}
For $G$ the interval $[b-1,b]$ is invariant, and if we glue together
the endpoints of this interval, $G$ will be the rotation of the
resulting circle by $b$ (that is, the angle $2\pi b$). Thus, if
$b=k/n$, $k,n$ coprime, then for any $x\in [b-1,b]$
\begin{equation}\label{Gi1n}
|G^i(x)-x|\ge\frac 1n
\end{equation}
for $i=1,\ldots,n-1$, and $G^n(x)=x$.
\end{remark}

Throughout the paper we set $b=k/n$, where $k,n$ are coprime natural numbers and $k<n$.

\section{Rotation-like behavior}\label{Section3}

Under assumptions of the preceding section, we will show the following lemma.

\begin{lemma}\label{le}
For every $F\in \mathfrak{F}$ there exists an attracting periodic orbit of $F$ of period $n$, that
lies in first and third lap. Both critical points of $F$ are in the
immediate basin of attraction of this periodic orbit.
\end{lemma}

\begin{proof}
Call a subinterval of $[b-1,b]$ \emph{basic} if it is of the form
$[i/n,(i+1)/n]$ for an integer $i$. Clearly, $G$ maps a basic interval
to a basic interval.

We will show that in the basic interval
$[0,1/n]$ there is a lap $[u,v]$ of $F^n$, in which there is an
attracting fixed point of $F^n$. Moreover, the images of this lap
under $F^i$ ($i=0,1,\dots,n-1$) are contained in basic intervals and
$F$ is increasing on them. Remember that we know the dynamics of $G$
and by (C1) and (C2), $F$ is very close to $G$ except in a small
neighborhood of $0$. We will be thinking of
$F-G$ as a ``correction'' that we add to $G$ (which is simple) to get
$F$. The correction closer to the critical points (but on the
increasing laps) is taken care by (C3) and (C4).

Set $u=y_+$. By~(C3), $u\in(0,\frac1{2n})$. We have
$F(y_+)-(b-1)=1-g(y_+)+y_+$ so $F(u)$ is in the basic interval
$[b-1,b-1+1/n]$, and its distance from the left endpoint of this
interval is $1-g(y_+)+y_+$. For several next iterates of $F$ the point
$F^i(u)$ is in some basic interval and, by conditions~(C), its distance from the left
endpoint of that interval is still less than $\frac1{2n}$.

This clearly continues until $F^i(u)$ gets to the basic interval
$[-1/n,0]$. However, we can continue it further, since $F^i(u)$ is
close to the left endpoint of this interval, and~(C2) works for
$|x|\ge\frac1{2n}$. Nevertheless, this moment is special, since in
$[-1/n,0]$ there is the critical point $y_-$ of $F$. In our
construction we should have $y_-=F^i(v)$. Thus, we should work
backward in time to recover $v$.

This is possible, because locally $G$ is of the form $G(z)-z+\alpha$
for some $\alpha$. Therefore,
\[
G^{-1}(x)-F^{-1}(x)=x-\alpha-F^{-1}(x)=F(F^{-1}(x))-G(F^{-1}(x)),
\]
and we can use (C1) and (C2). Clearly, we can also go forward from $F^i(v)$. By
the same arguments as before, and using (C3) and (C4), we see that the
sets $F^j([u,v])$, $j=0,1,\dots,n$, are all contained in basic
intervals, and in the first and third lap of $F$.

Since $G^n(x)=x$ for every $x$, we have $F^n(u)-u=F^n(u)-G^n(u)$, and
this number is equal to the sum of ``corrections'' along the
trajectory of $u$ of length $n$. One of those corrections is
$1-g(y_+)$, and the rest of them are less than $\eps$ in absolute
value. Therefore, by~(C3), $F^n(u)-u>0$. Similarly, when we estimate
$F^n(v)-v$, we get $F^n(v)-v<0$.

Thus, there is a fixed point $z$ of $F^n$ in $(u,v)$. We know that
$F'<1$, so the orbit of $z$ is attracting and both $u$ and $v$ (and
therefore of both critical points of $F$) are in the immediate basin
of attraction of the orbit of $z$.
\end{proof}

\begin{lemma}\label{lee}
Under the assumptions of the preceding lemma, the attracting periodic
orbit attracts trajectories of Lebesgue almost all points from $\R$.
\end{lemma}

\begin{proof}
Let $z$ be the attracting periodic orbit of $F$ of period $n$ from Lemma \ref{le}.
The map $F$ has negative Schwarzian derivative by the assumption.
Moreover, we can compactify $\R$ by adding points at plus and minus
infinity, and those points are repelling fixed points. It turns out
that from this and from what we already proved it follows that the
trajectories of Lebesgue almost all points of $\R$ are attracted to
the orbit of $z$.

The proof of this fact is almost identical to the proof of Theorem~4.1
of~\cite{M2} (whose main part is taken from the proof of Theorem~1.3
of~\cite{M1}). The reader can find the details there. Here we will
only sketch the main idea of the proof.

Since the proofs we mentioned are for maps of a bounded interval into
itself, we will work with the map $f:[0,1]\to[0,1]$, smoothly
conjugate to $F$. % (for instance, by $x\mapsto e^x/(1+e^x)$). 
Then the basic properties of $F$ and $f$ are the same. Moreover, the points $0$
and $1$ are repelling fixed points of $f$. We will denote by $Z$ the
periodic orbit of $f$ corresponding to the orbit of $z$ for $F$.

Let $U$ be the union of the immediate basin of attraction of $Z$ and a
small neighborhood of $\{0,1\}$. Then we define $W$ as a small open
neighborhood of $[0,1]\setminus U$, so that its closure is disjoint
from $Z$ and does not contain the critical points of $f$.

An interval $J\subset[0,1]$ is called a \emph{homterval} if for every
$n>0$ $f^n$ restricted to $J$ is a homeomorphism onto its image. Then
we can use Theorem~1.2 of~\cite{M1} to deduce that for every homterval
$J$ there exists $m\ge 0$ such that the closure of $f^m(J)$ is
contained in $U$. Now we can use Theorem~1.3 of~\cite{M1}. Note that
the assumption that $f$ has no sinks really means that there are no
sinks in $W$, and the assumption on the negative Schwarzian derivative
can be replaced by the assumption that $f$ is smoothly conjugate to a
map with negative Schwarzian derivative (the Lebesgue measure
$\lambda$ can be replaced by the measure $\mu$ equivalent to the
Lebesgue measure, transported via the conjugacy).

Thus, we see that on $W$ some iterate of $f$ is expanding (perhaps for
$\mu$, not $\lambda$). Now, the standard argument, as in the proof of
Theorem~4.1 of~\cite{M2} shows that the trajectory of $\mu$-almost
every point (and therefore of $\lambda$-almost every point) sooner or
later enters $U$. Then such trajectory converges to $Z$.
\end{proof}

An immediate consequence of Lemma \ref{le} and Lemma \ref{lee} is the following theorem.
\begin{theorem} \label{thmpern}
If $b=k/n$, where $k,n\in\N$ are coprime, then for every $F\in\mathfrak{F}$, there exists an attracting periodic orbit of $F$ of period $n$, which attracts trajectories of Lebesgue almost all points from $\R$. This attracting periodic orbit   lies in first and third lap, in the invariant interval $[b-1,b]$, and  it behaves like an orbit of $G$. That is, if we glue $b-1$ with $b$ to get a circle, the order of the points on the orbit is the same as for the rotation by $b$.
\end{theorem}

\section{EOS maps}\label{seceos}

Now we will focus on the concrete family of maps from $\mathfrak{F}$.
Let us consider the family of EOS maps $F\colon \R\to\R$, given by
\begin{equation}\label{e2}
F(x)=x+b-\frac1{e^{-ax}+1},
\end{equation}
parametrized by two parameters $a>0$ and $b\in (0,1)$. A family 
of maps equivalent to \eqref{e2} was first observed in \cite{eirola1996chaotic}.
This family describes population dynamics in mathematical biology:
%%This family arises from problems in mathematical biology: %(two predator-one-prey model) \cite{eirola1996chaotic,kryzhevich2021bistability}:
in \cite{kryzhevich2021bistability} a one-dimensional version of two-predator-one-prey model is introduced by a family of maps $f_{PP}\colon \R \to \R$ given by
    \[f_{PP}(x)=x+B-\frac{k}{1+e^x}.\]
    By taking $x=kz$ we get
\[f_{PP}(kz)=k\left(z+\frac Bk-\frac{1}{1+e^{kz}}\right)=k\left(z+b-\frac{1}{e^{-az}+1}\right)=kF(z)\] for $a=-k$ and $b=\frac Bk$, so $f_{PP}$ is conjugate to $F$.

Moreover, it is also conjugate to the map from machine learning %artificial intelligence 
and game theory. That is, %\cite{chotibut2021family}: 
the evolution in a simple population game where agents are using multiplicative weights algorithm, established algorithm from machine learning and economics \cite{Arora05themultiplicative},  is described by the family of maps
$f_{MW}\colon [0,1]\to [0,1]$, given by
\begin{equation}\label{e2f}
f_{MW}(y)=\frac{y}{y+(1-y)\exp(a(y-b))},
\end{equation}
where $a>0$, $b\in (0,1)$ (for a thorough discussion, see \cite{bielawskiheterogeneity,chotibut2021family}). 
 By taking $x=\frac 1a \log\frac{y}{1-y}$ and \[F(x)=\frac 1a
\log \frac{f_{MW}(y)}{1-f_{MW}(y)}\] we see that the family $F$ is conjugate to $f_{MW}$ on $(0,1)$.

\begin{comment}
\footnote{
The evolution in a simple population game where agents are using multiplicative weights algorithm, established algorithm from machine learning and economics,  is described by the map
$f_{MW}\colon [0,1]\to [0,1]$, where
\begin{equation}\label{e2f}
f_{MW}(y)=\frac{y}{y+(1-y)\exp(a(y-b))},
\end{equation}
where $a>0$, $b\in (0,1)$ (for a thorough discussion, see \cite{chotibut2021family,bielawskiheterogeneity}). 
 By taking $x=\frac 1a \log\frac{y}{1-y}$ and \[F(x)=\frac 1a
\log \frac{f_{MW}(y)}{1-f_{MW}(y)}\] we see that the map $F$ is conjugate to $f_{MW}$ on $(0,1)$.} 
as well as mathematical biology (two
predator-one-prey model)
\cite{eirola1996chaotic,kryzhevich2021bistability}.\footnote{
    In \cite{kryzhevich2021bistability} a one-dimensional version of two-predator-one-prey model is introduced by a map $f_{PP}\colon \R \to \R$ given by
    \[f_{PP}(x)=x+B-\frac{k}{1+e^x}.\]
    By taking $x=kz$ we get
\[f_{PP}(kz)=k\left(z+\frac Bk-\frac{1}{1+e^{kz}}\right)=k\left(z+b-\frac{1}{e^{-az}+1}\right)=kF(z)\] for $a=-k$ and $b=\frac Bk$, so $f_{PP}$ is conjugate to $F$.

}
\end{comment}

For the family of EOS maps we have $g(x)=\frac1{e^{-ax}+1}$. Observe that
\[
g(-x)+g(x)=\frac{e^{-ax}}{e^{-ax}+1}+\frac1{e^{-ax}+1}=1,
\]
so, by Remark \ref{r1-sym}, the simplifications we mentioned apply.

Let us fix $b$ as in the preceding section. We will show for $a$ sufficiently large, $F$ satisfies our
assumptions.

%Let us make some calculations. We have
%\begin{equation}\label{e3}
%g'(x)=\frac{at}{(t+1)^2},\ \ \textrm{where}\ \ t=e^{-ax}.
%\end{equation}
%Therefore, $y_+$ and $y_-$ are given by $\xi(t)=0$, where
%$\xi(t)=(t+1)^2-at$. Denote the corresponding values of $t$ by
%$t_-,t_+$. They exist if $a>4$. Moreover, (A) follows from~\eqref{e3}.

For $x\ne 0$ we have by~\eqref{en3}
\begin{equation}\label{e5}
F(x)-G(x)=\frac{\sgn(x)}{e^{a|x|}+1},
\end{equation}
where $\sgn(x)$ is the sign of $x$. Therefore,
\begin{equation}\label{e6}
\textrm{if}\ \ \ |x|\ge\delta,\ \ \ \textrm{then}\ \ \ |F(x)-G(x)|\le
e^{-a\delta},
\end{equation}
which is exponentially small in $a$. Thus, we get (C1) for sufficiently
large $a$.

As $a$ goes to infinity, $g(b-1)$ goes to 0, and $g(b)$ goes to 1.
Therefore, \eqref{bb} holds for sufficiently large $a$.

Let us make some calculations. We have
%\begin{equation}\label{e3}
\[
g'(x)=\frac{at}{(t+1)^2},\ \ \textrm{where}\ \ t=e^{-ax}.
\]
%\end{equation}
Therefore, $y_+$ and $y_-$ are given by $\xi(t)=0$, where
$\xi(t)=(t+1)^2-at$. Denote the corresponding values of $t$ by
$t_-,t_+$. They exist if $a>4$. %Moreover, (A) follows from~\eqref{e3}.
%Let us make estimates for $y_-$ and $y_+$. They correspond to the
%roots of $\xi(t)=(t+1)^2-at$, that is $t_-$, $t_+$. 
We have $\xi(a-2)=1>0$ and $\xi(a-3)=4-a<0$. Therefore, $a-3<t_-<a-2$.
Since $y_-=-\frac1a\log t_-$ and $y_+=-y_-$, we get
\begin{equation}\label{e4}
\frac1a\log(a-3)<y_+<\frac1a\log(a-2).
\end{equation}
Therefore, for sufficiently large $a$, $y_+\in(0,b)$ and we get \eqref{aa}.

From~\eqref{en3} and~\eqref{e4} we
get
\begin{equation}\label{e7}
\frac1{a-1}<F(y_+)-G(y_+)<\frac1{a-2}.
\end{equation}
Therefore, if $\eps$ is sufficiently small and $a$ is sufficiently
large, we get (C3).

Finally, the fact that the Schwarzian derivative of $F$ is negative,
was proved in~\cite{bielawski2022memory}.
Thus, the family of EOS maps defined by \eqref{e2} belongs to $\mathfrak{F}$ for sufficiently large $a$. Therefore, Theorem~\ref{thmpern}
applies.

\begin{figure}[h!]
\begin{center}
\includegraphics[width=80truemm]{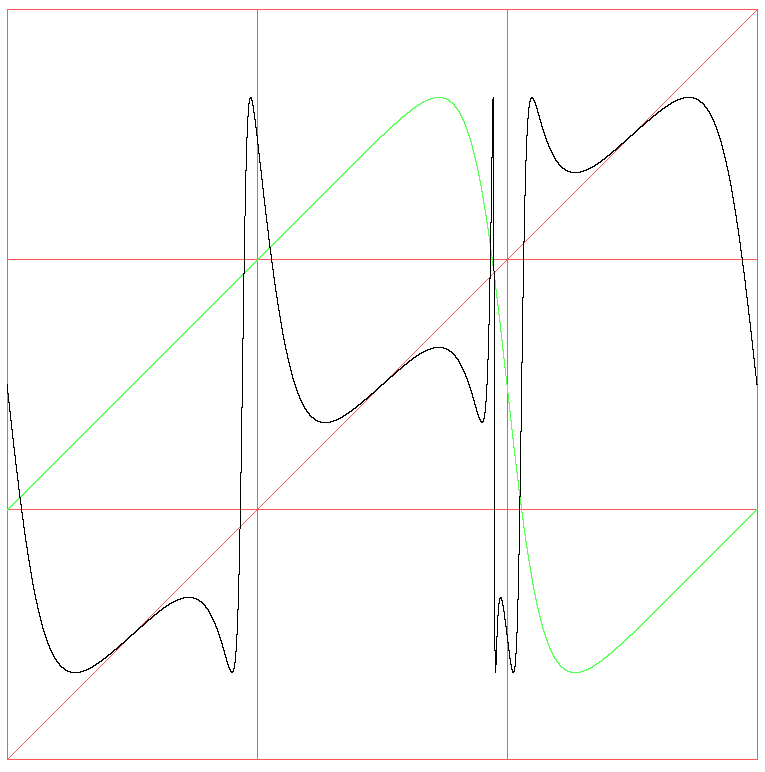}
\caption{The maps $F$ (green) and $F^3$ (black) defined by \eqref{e2} for $b=1/3$ and
  $a=40$. Horizontal and vertical red lines are at the levels $b-1$,
  $b-2/3$, $b-1/3$ and $b$.}\label{f1}
\end{center}
\end{figure}

\begin{remark}
Other families of maps, which fulfill assumptions of Theorem \ref{thmpern}  can be obtained for instance by taking $g(x)=\frac 12+\frac{1}{\pi}\arctan (ax)$ or $g(x)=\frac 12+\frac{1}{\sqrt{\pi}}\int_0^{ax}e^{-t^2}dt$ for sufficiently large $a$.
\end{remark}

\begin{comment}
In the remaining part of the paper we will focus on the dynamics of the map \eqref{e2} for different values of $a$.
First, our map $F$ has exactly one fixed point,
\begin{equation}\label{fpoint}
c=\frac1a\log\frac{b}{1-b}.
\end{equation}
 By the conjugacy argument and result from  \cite{chotibut2021family}, for $a\leq
\frac{2}{b(1-b)}$ the point $c$ is globally attracting. Thus, we will
assume that $b\in(0,1)$ and $a>\frac{2}{b(1-b)}$. Let us make some
calculations. We have
\begin{equation}\label{e3}
F'(x)=1-\frac{at}{(t+1)^2},\ \ \textrm{where}\ \ t=e^{-ax}.
\end{equation}
Therefore, the critical points of $F$ are given by $\xi(t)=0$, where
$\xi(t)=(t+1)^2-at$. Denote those critical points by $y_-,y_+$. They exist if $a>4$.    
\end{comment}

\section{Bifurcations for a rational $b$ for EOS maps}\label{secbif}

Now we describe other phenomena that happen when $b$ is rational and
$a$ increases for the family of EOS maps defined by \eqref{e2}. We make computations for $b=k/11$, where $k=1,2,3,4,5$.
Note that by the symmetry, the behavior for $b$ replaced by $1-b$ is
the same.

We know already that if $b=k/n$ in lowest terms, then for $a$
sufficiently large there exists an attracting periodic orbit of $F$ of
period $n$, that lives on the first and third lap, has $k$ points on
the third lap, and attracts Lebesgue almost all points.  In the proof of Theorem \ref{thmpern}
at an early stage the dependence on $k$ vanishes, but it is not clear
whether the estimates how large $a$ should be are really mostly
independent of $k$, or this is only a feature of the concrete proof.

Our numerical results suggest the former.
Figure \ref{fig-bif} shows the bifurcation diagram for
$b=k/11$, where $k=1,2,3,4,5$. On the horizontal axis there is $a$,
from $100$ to $180$; the vertical cyan lines are every $10$.
Horizontal red lines indicate the positions of the critical points.
When the attracting periodic orbit of period $11$ is born, it has two
points in the second lap. Then, as $a$ increases, those points move to
the first and third lap, and the situation described in %the theoretical part 
Section \ref{Section3} is created. It seems that they move through the
critical points in more or less the same time, independently of $k$. 
% In fact, {\bf the pictures suggest that the important bifurcation values of $a$ do not depend on $k$.}

\begin{figure}[h]
\begin{subfigure}{.48\textwidth}
  \includegraphics[width=1.0\textwidth]{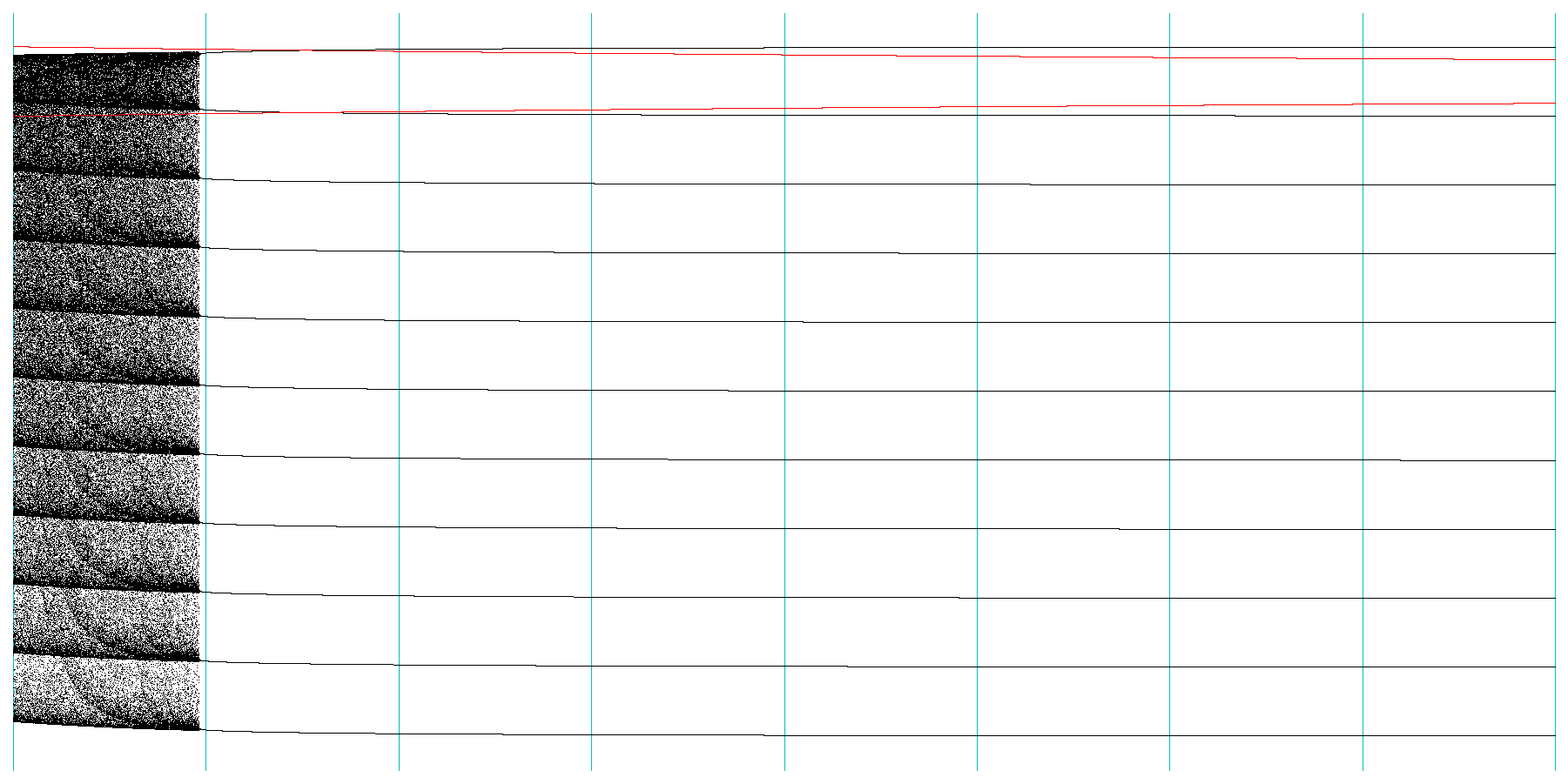}
\end{subfigure}%
\hfill
\begin{subfigure}{.48\textwidth}
\includegraphics[width=1.0\textwidth]{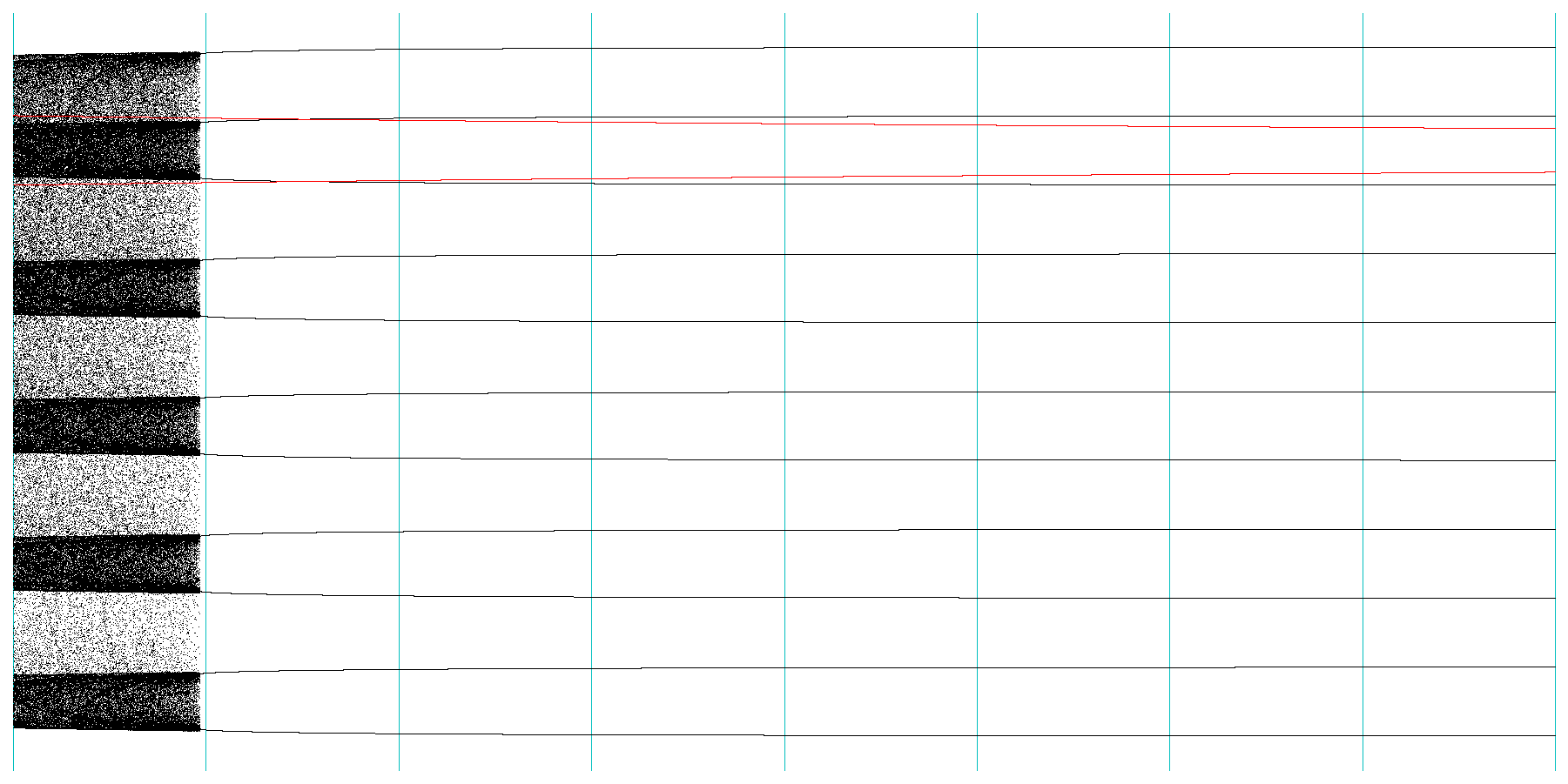}
\end{subfigure}
\begin{subfigure}{.48\textwidth}
  \includegraphics[width=1.0\textwidth]{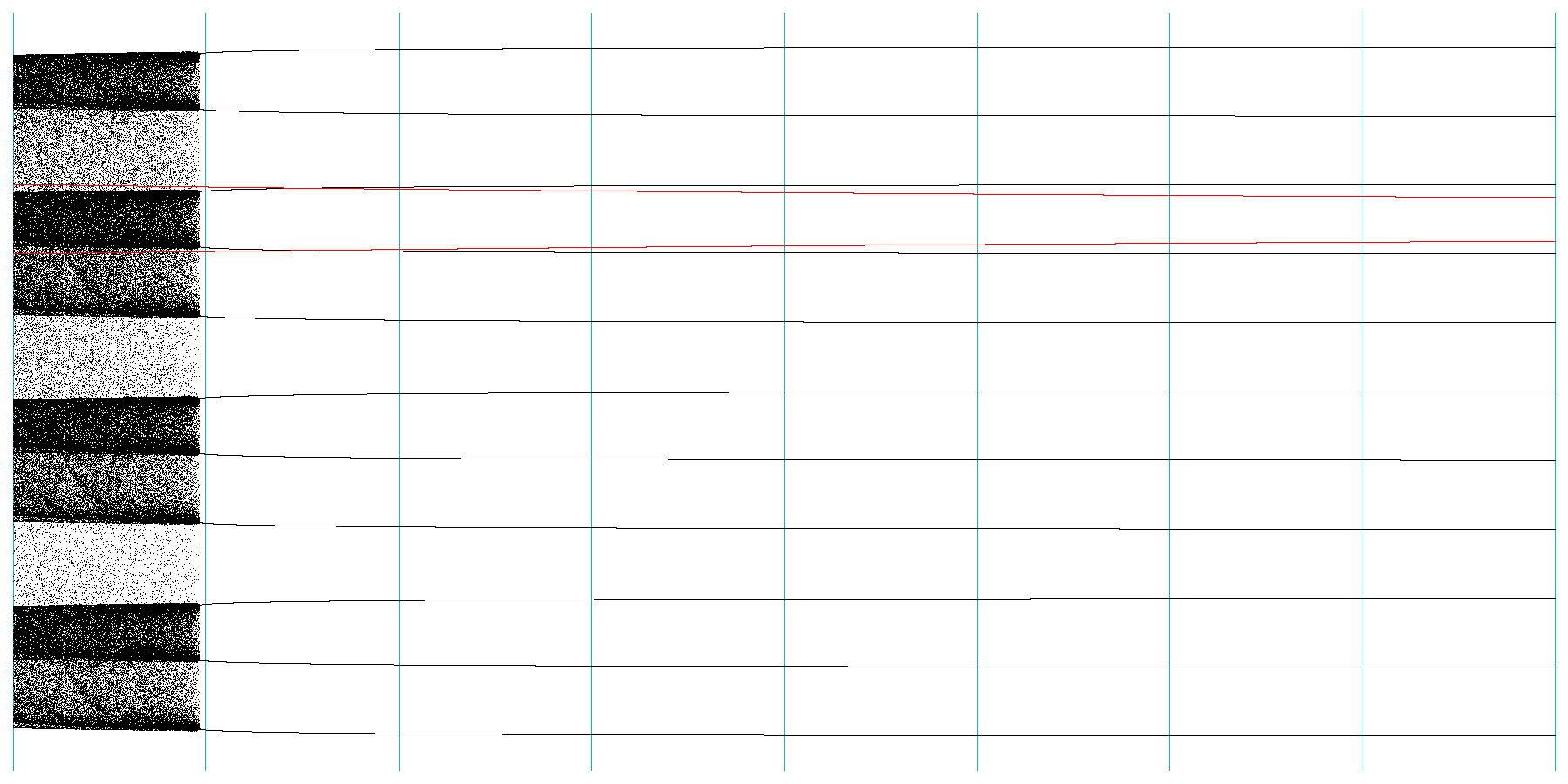}
\end{subfigure}%
\hfill
\begin{subfigure}{.48\textwidth}
  \includegraphics[width=1.0\textwidth]{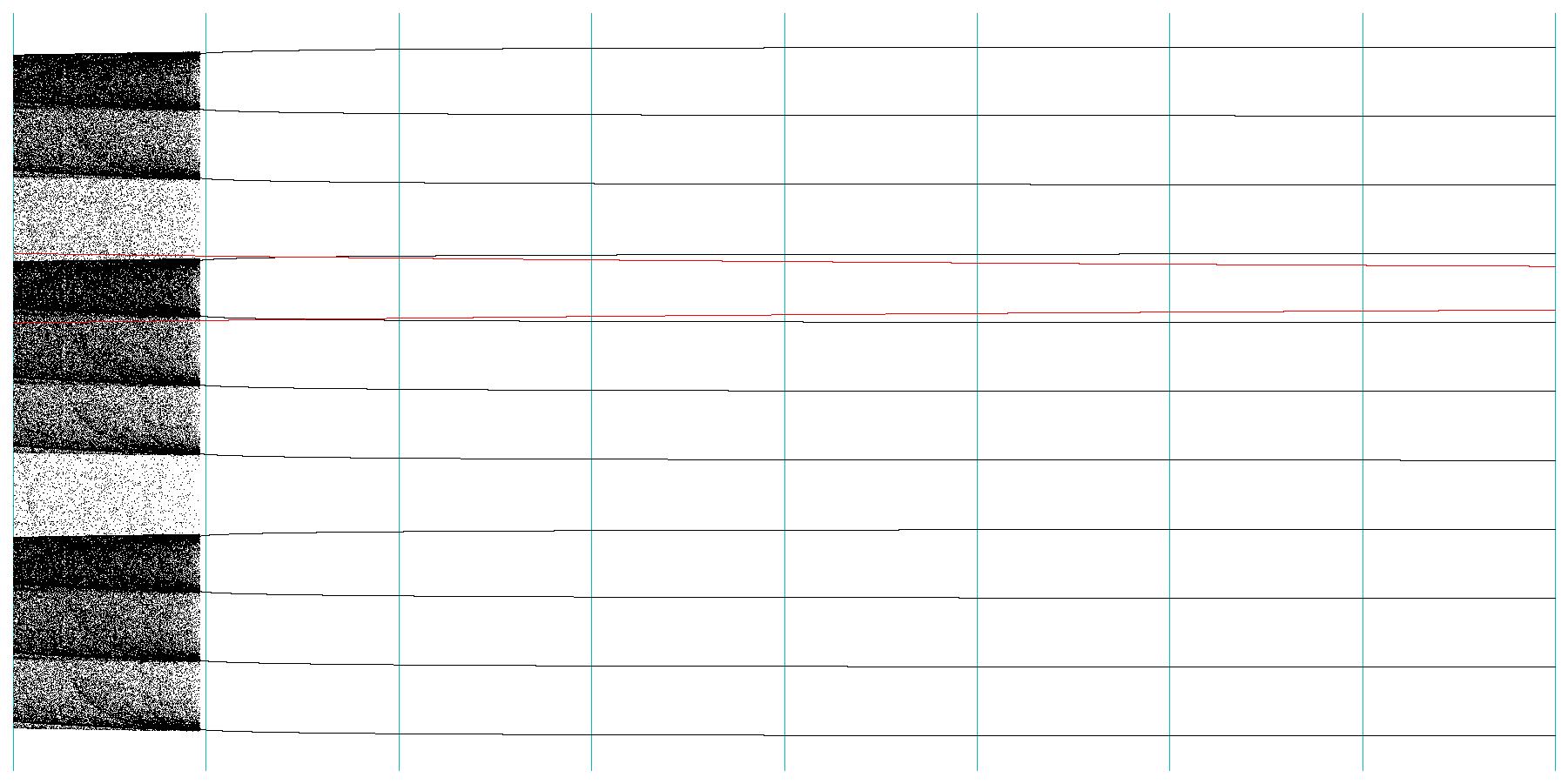}
\end{subfigure}
\hfill
\begin{subfigure}{.48\textwidth}
  \includegraphics[width=1.0\textwidth]{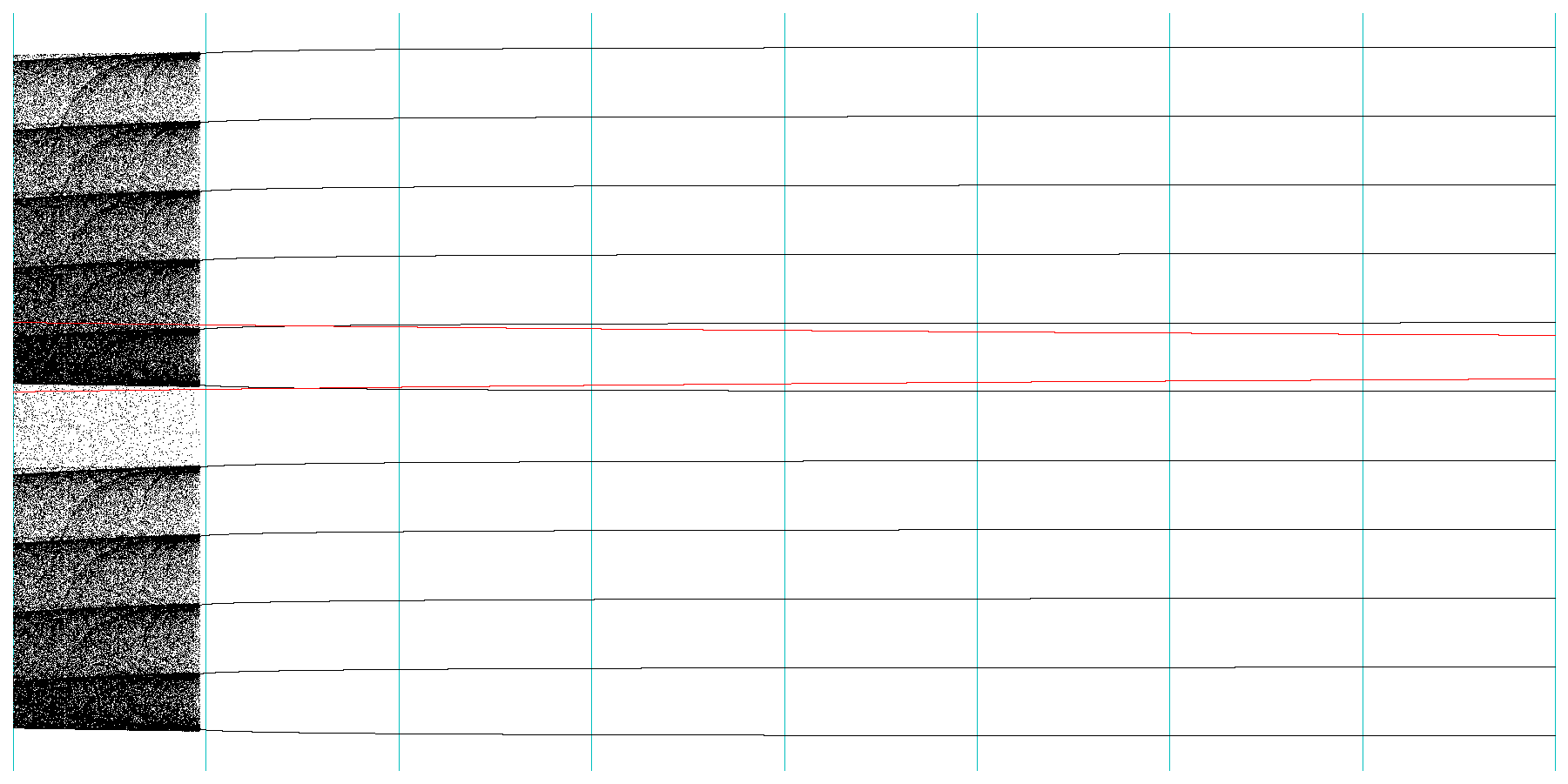}
\end{subfigure}%
\hfill
\begin{subfigure}{.48\textwidth}
  
\end{subfigure}%
%\begin{minipage}{.5\textwidth}
% 
%\end{minipage}%
 \caption{Bifurcation diagrams for EOS maps
$b=k/11$, where $k=1,2,3,4,5$. On the horizontal axis there is $a$,
from $100$ to $180$; the vertical cyan lines are every $10$.
Horizontal red lines indicate the positions of the critical points.}
\label{fig-bif}

\end{figure}

\begin{comment}
\begin{figure}[h!]
\centering
    \begin{subfigure}{1.0\textwidth}
\begin{center}
\includegraphics[width=0.77\textwidth]{bbbb1over11.jpg}
\caption{ $b=1/11$.}\label{bbbb1}
\end{center}
\end{subfigure}
\begin{subfigure}{1.0\textwidth}
\begin{center}
\includegraphics[width=0.77\textwidth]{bbbb2over11.jpg}
\caption{ $b=2/11$.}\label{bbbb2}
\end{center}
\end{subfigure}
\begin{subfigure}{1.0\textwidth}
\begin{center}
\includegraphics[width=0.77\textwidth]{bbbb3over11.jpg}
\caption{ $b=3/11$.}\label{bbbb3}
\end{center}
\end{subfigure}
\begin{subfigure}{1.0\textwidth}
\begin{center}
\includegraphics[width=0.75\textwidth]{bbbb4over11.jpg}
\caption{ $b=4/11$.}\label{bbbb4}
\end{center}
\end{subfigure}
\begin{subfigure}{1.0\textwidth}
\begin{center}
\includegraphics[width=0.75\textwidth]{bbbb5over11.jpg}
\caption{ $b=5/11$.}\label{bbbb5}
\end{center}
\end{subfigure}
\caption{Bifurcation diagrams for
$b=k/11$, where $b=1,2,3,4,5$. On the horizontal axis there is $a$,
from $109$ to $118$; the vertical cyan lines are every $0.5$.
Horizontal green lines indicate the positions of the critical points.}
\end{figure}
\end{comment}

This independence of $k$ (as well as other aspects, %of this independence, 
like the position of certain points of the attracting
periodic orbit) can be explained in the following way.  Of course, it
is not a complete independence, but approximate one. %{\cre A rigorous proof would have to include estimates.}

\vspace{0.5cm}

{\bf Approximate independence of important bifurcations of the numerator.} Let us now take $b=k/n$, where $k<n$ are coprime positive integers and assume that $n>3$. %(or maybe even $n>5$).
Since we want to vary $k$, it is reasonable to assume that $n$ is a prime
number. As
\[
-F(-x)=x+(1-b)-\frac1{e^{-ax}+1},
\]
it is enough to look at $k<n/2$. Then we take a reasonably large $a$.
For the sake of illustration, we choose $n=11$ and $a=110$.

Now we replace $F$ by a family of maps $\wf$, defined as $F$ on $[-1/n,1/n]$ and
$G$ on the rest of the real line. By~\eqref{e6}, it is very close to
$F$ (indistinguishable on the pictures).

Let us divide the real line into intervals $I_j=[(j-1)/n,j/n]$ (we do
not care about endpoints). If $j\ne0,1$, then our family $\wf$ is a
translation on $I_j$ and maps it onto some $I_i$. The situation is
more difficult if $j$ is 0 or 1. Then we take a subinterval $J_j$ of
$I_j$ on which the image under $\wf$ is contained in the same $I_i$ as
the image under $G$. Next we take the smallest $r_j>0$ for which
$\wf^{r_j}(J_j)\subset I_{1-j}$. Because the family $\wf$ is just a
translation on each $I_i$ with $i\ne0,1$, the graph of $\wf^{r_j}$ on
$J_j$ is just a vertical translation (by a fraction with the
denominator $n$) of the graph of $F$ on the same interval.

In Figures~\ref{4a1},~\ref{4a3}, and~\ref{4a4}, we see the graphs of
$F$ on $[b-1,b]$ for $b=1/11,3/11,4/11$ respectively. Moreover, in
blue we see the graphs of $F^{r_j}$ on $J_j$, and in magenta the
graphs of $F^n$ on slightly smaller intervals ($K_j$).

\begin{figure}[h!]
\centering
\begin{subfigure}{.45\textwidth}
\begin{center}
\includegraphics[width=.9\linewidth]{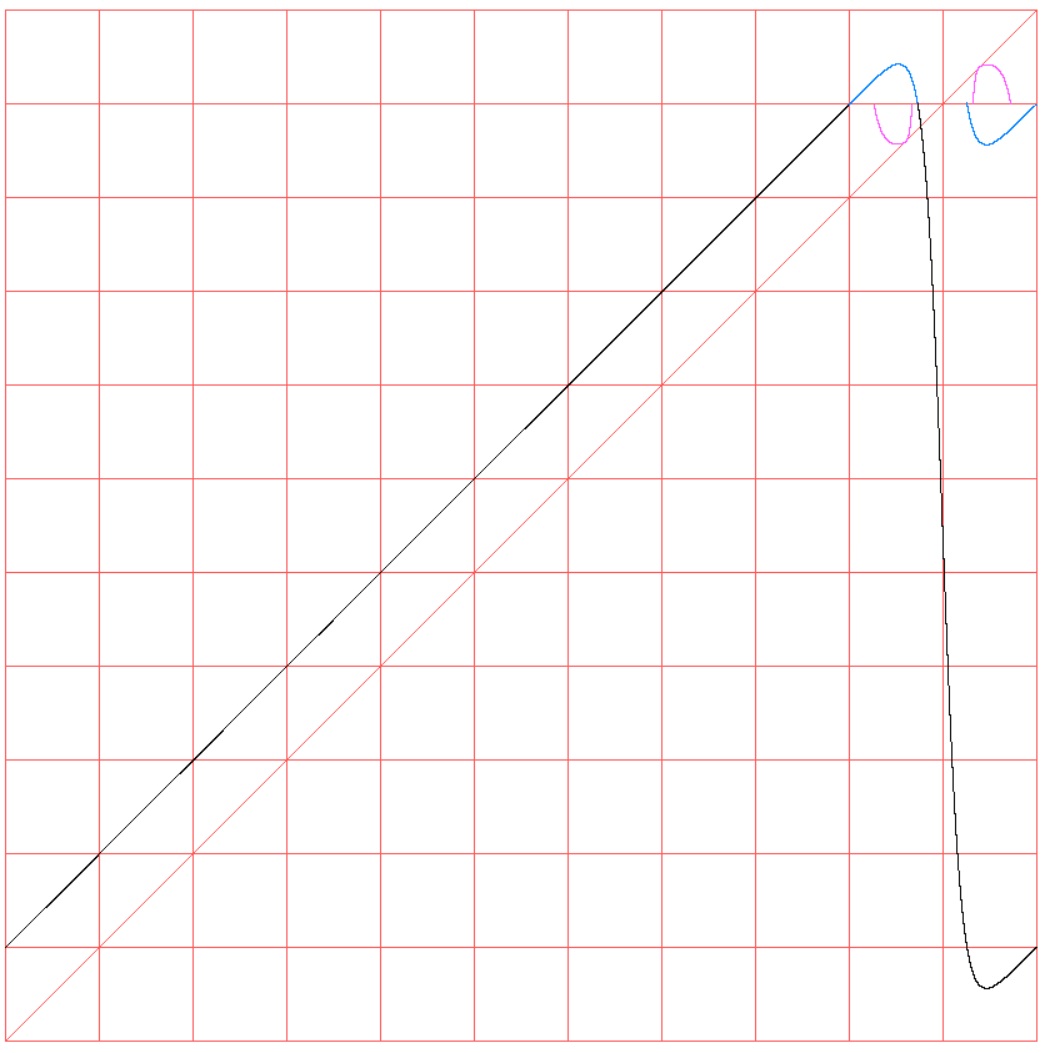}
\caption{Here $b=1/11$.}\label{4a1}
\end{center}
\end{subfigure}
\begin{subfigure}{.45\textwidth}
\begin{center}
\includegraphics[width=.9\linewidth]{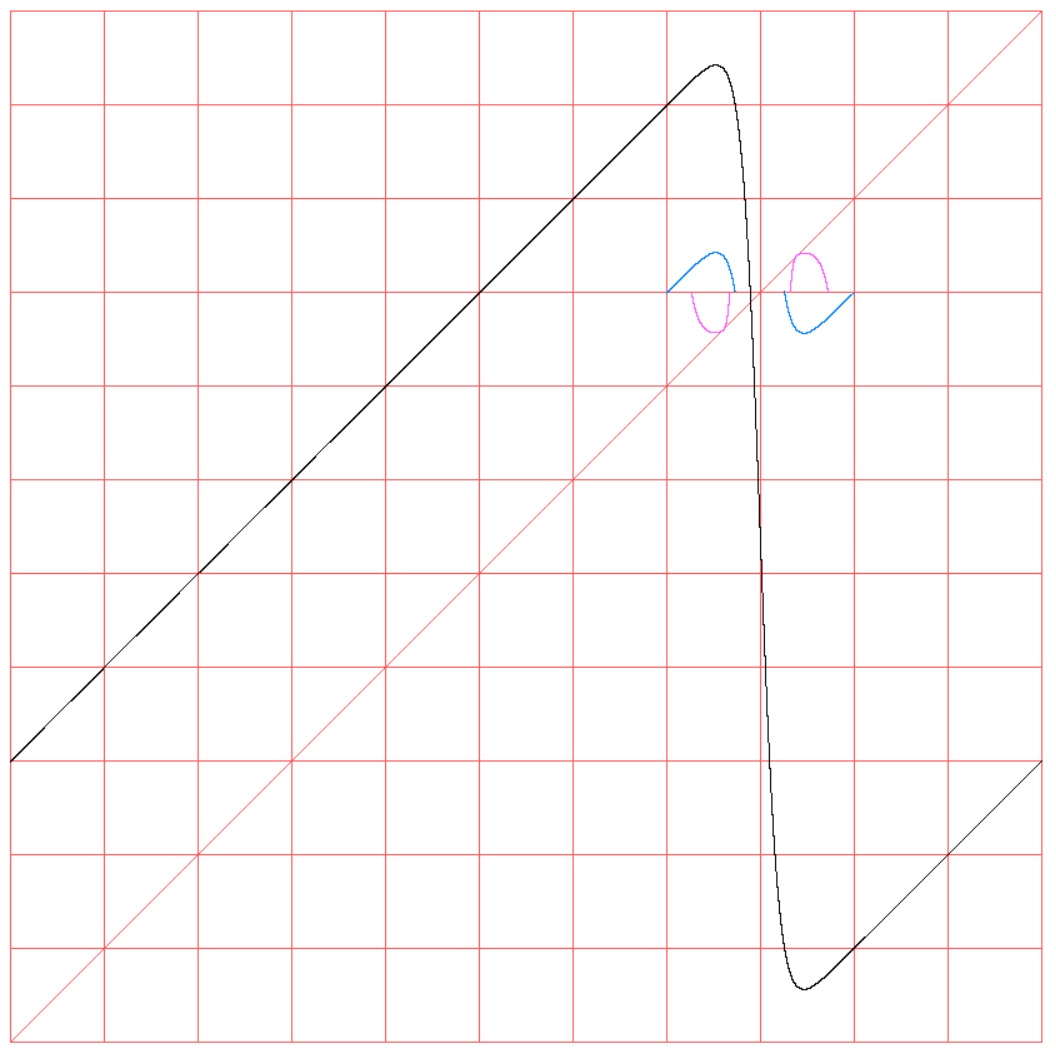}
\caption{Here $b=3/11$.}\label{4a3}
\end{center}
\end{subfigure}
\begin{subfigure}{.45\textwidth}
\begin{center}
\includegraphics[width=.9\linewidth]{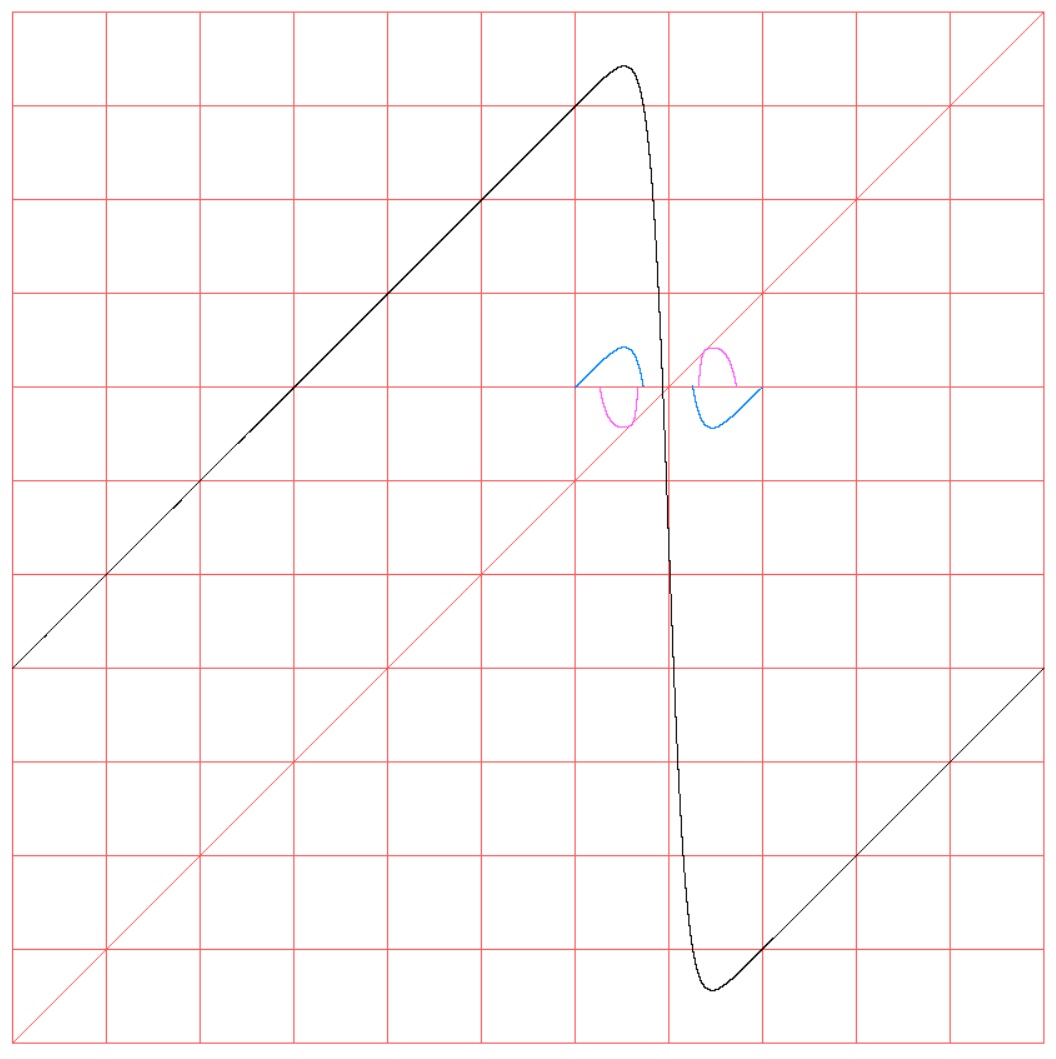}
\caption{Here $b=4/11$.}\label{4a4}
\end{center}
\end{subfigure}
\caption{Graphs of an EOS map $F$ for $a=110$ on $[b-1,b]$. In blue we see the graphs of
  $F^{r_j}$ on $J_j$, and in magenta the graphs of $F^n$ on $K_j$.}
  \label{fig-a}
\end{figure}

When we vary $k$, the graph of $\wf$ (and of $F$) just moves up or
down. Therefore the graphs of $\wf^{r_j}$ on $J_j$ ($j=0,1$) stay the
same. This is another way of saying that the maps $\wf^{r_j}$ on $J_j$
stay the same. To get $\wf^n$, we have to compose the maps $\wf^{r_0}$
on $J_0$ and $\wf^{r_1}$ on $J_1$ (clearly, $r_0+r_1=n$). It makes
sense to look at $\wf^n$ on slightly smaller intervals $K_j\subset
J_j$, where $\wf^n(K_j)\subset K_j$. In our figures those graphs are
marked in magenta. By what we said, they are independent of $k$.

There is a small difference between $F^n$ and $\wf^n$ on those
intervals $K_j$, but it seems that it is so small, that is practically
invisible. Thus, for instance, existence of an attracting fixed point
of $\wf^n$ on $K_1$ and its position (those are independent of $k$) is
almost the same as the existence of an attracting fixed point of $F^n$
on $K_1$ and its position. This explains why the value of $a$ for
which an attracting periodic orbit of period $n$ is born, and the
position of the smallest positive point of this orbit, are practically
independent of $k$.

The explanation of the fact that the values of $a$ for which the left
or right critical point is periodic of period $n$ are almost the same
for the left and right critical points and almost independent of $k$,
is similar. Namely, one can easily check that $F(x)+F(-x)=2b-1$, and
it follows that $J_1=-J_0$ and for $x\in J_0$ we have
$\wf^{r_1}(-x)=-\wf^{r_0}(x)$. Therefore, for $x\in K_0\cup K_1$ we
have $\wf^n(-x)=-\wf^n(x)$. This shows independence of left-right for
$\wf$, and since $\wf^n$ on $K_0\cup K_1$ is independent of $k$, also
independence of those values of $a$ on $k$. Since $F$ is very close to
$\wf$, we get for $F$ almost independence.
\medskip

{\bf Farey neighbors.\footnote{The reader not familiar with the Farey tree,
can find its description and terminology for instance in \cite{Devaney1999}.}}
In the bifurcation diagrams for $b=k/11$ (we observe similar
phenomenon for other denominators, too) we see another interesting
thing. %Figures~\ref{bbb1}--\ref{bbb5} are like Figure~\ref{fig-a},
Figure \ref{bbb} is like Figure~\ref{fig-bif},
but $a$ varies from $80$ to $89$. We
see in the bifurcation diagram periodic windows at the approximately
same place for each $k$. For those periodic orbits
one point is very close (for some value of $a$ equal) to one of the
critical points. However, there is another point of the orbit which is
in the second lap. All other points of the orbit are in the first and
third lap. If we look at the period $q$ of the orbit and the number
$p$ of the points in the third lap (plus the point which is maybe in
the second lap, but is close to a critical point), then the
fraction $p/q$ is a Farey neighbor of $k/11$. In fact, it is one of
the Farey parents of $k/11$, and the one with the larger denominator.
For instance, $\frac3{11}=\frac{1+2}{4+7}$ and $2\cdot4-1\cdot7=1$, so
the Farey parents of $3/11$ are $1/4$ and $2/7$; what we see in
Figure~\ref{bbb} with $k=3$ is $p/q=2/7$.

\begin{figure}[h!]
\centering
\begin{subfigure}{0.48\textwidth}
\begin{center}
\includegraphics[width=1.0\textwidth]{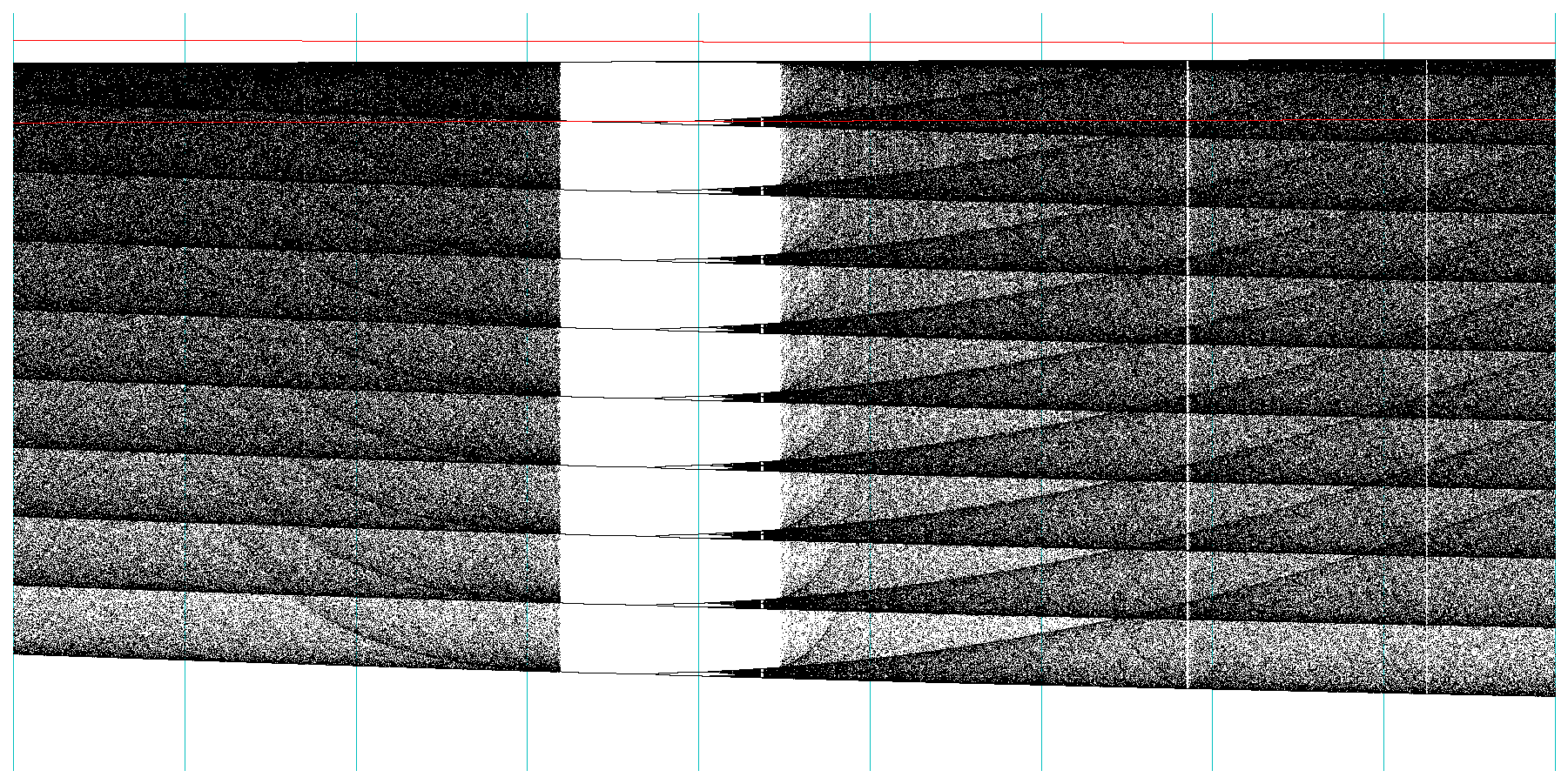}
%\caption{Here $b=1/11$.}\label{bbb1}
\end{center}
\end{subfigure}
\hfill
\begin{subfigure}{0.48\textwidth}
\begin{center}
\includegraphics[width=1.0\textwidth]{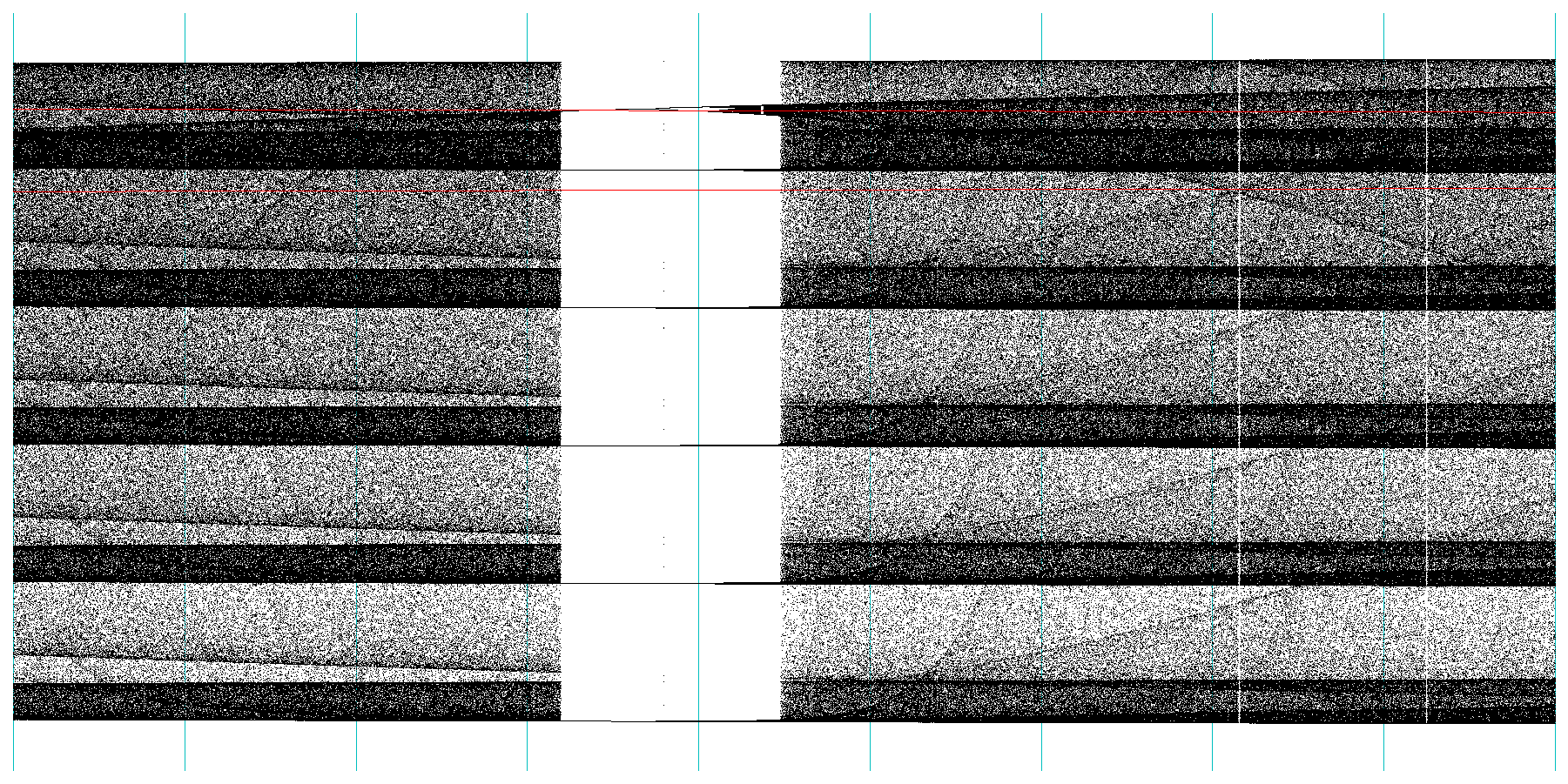}
%\caption{Here $b=2/11$.}\label{bbb2}
\end{center}
\end{subfigure}
\hfill
\begin{subfigure}{0.48\textwidth}
\begin{center}
\includegraphics[width=1.0\textwidth]{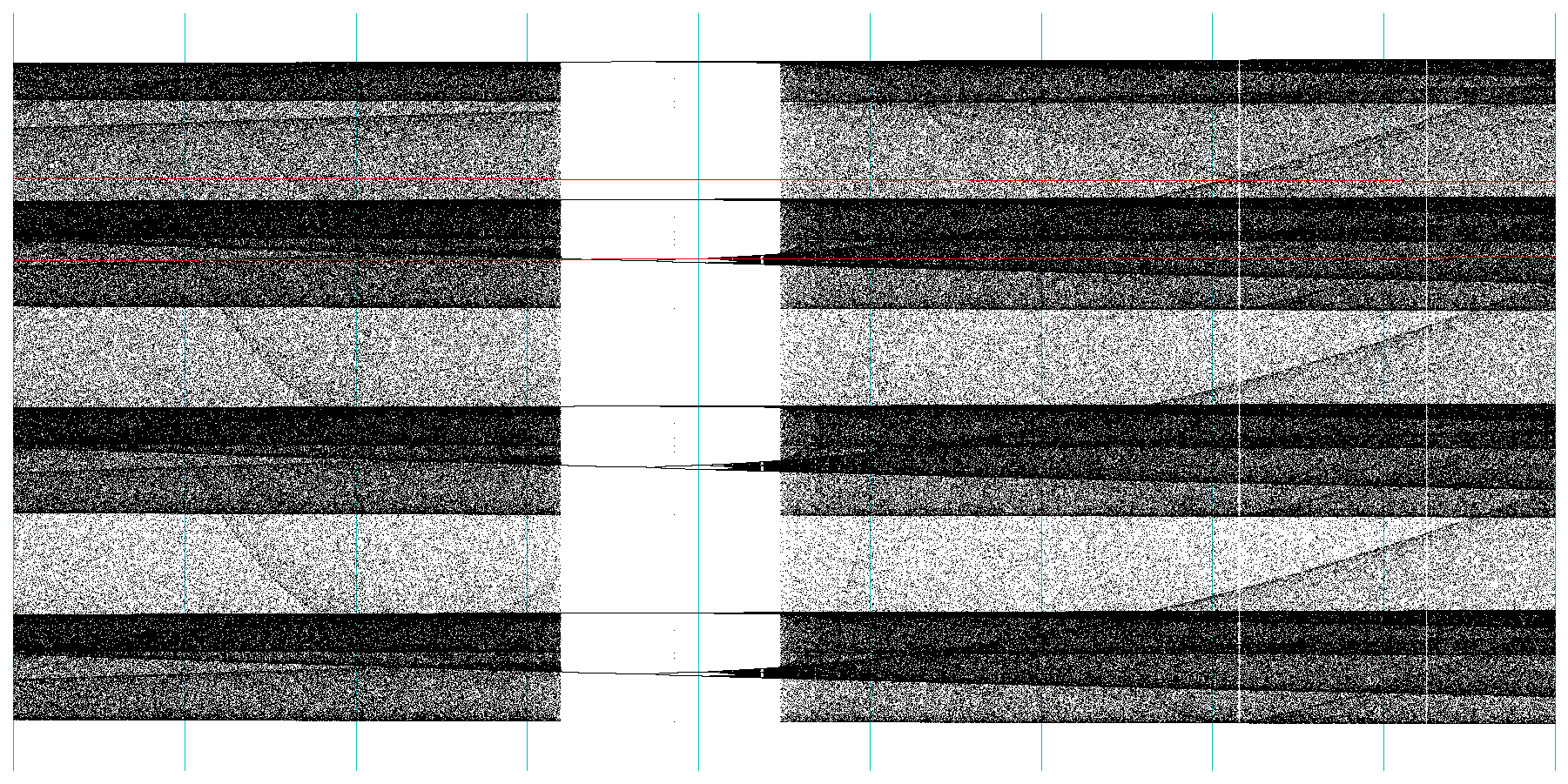}
%\caption{Here $b=3/11$.}\label{bbb3}
\end{center}
\end{subfigure}
\hfill
%\caption{Bifurcation diagrams}
%\end{figure}
%
%\begin{figure}[h!]
%\centering
\begin{subfigure}{0.48\textwidth}
\begin{center}
\includegraphics[width=1.0\textwidth]{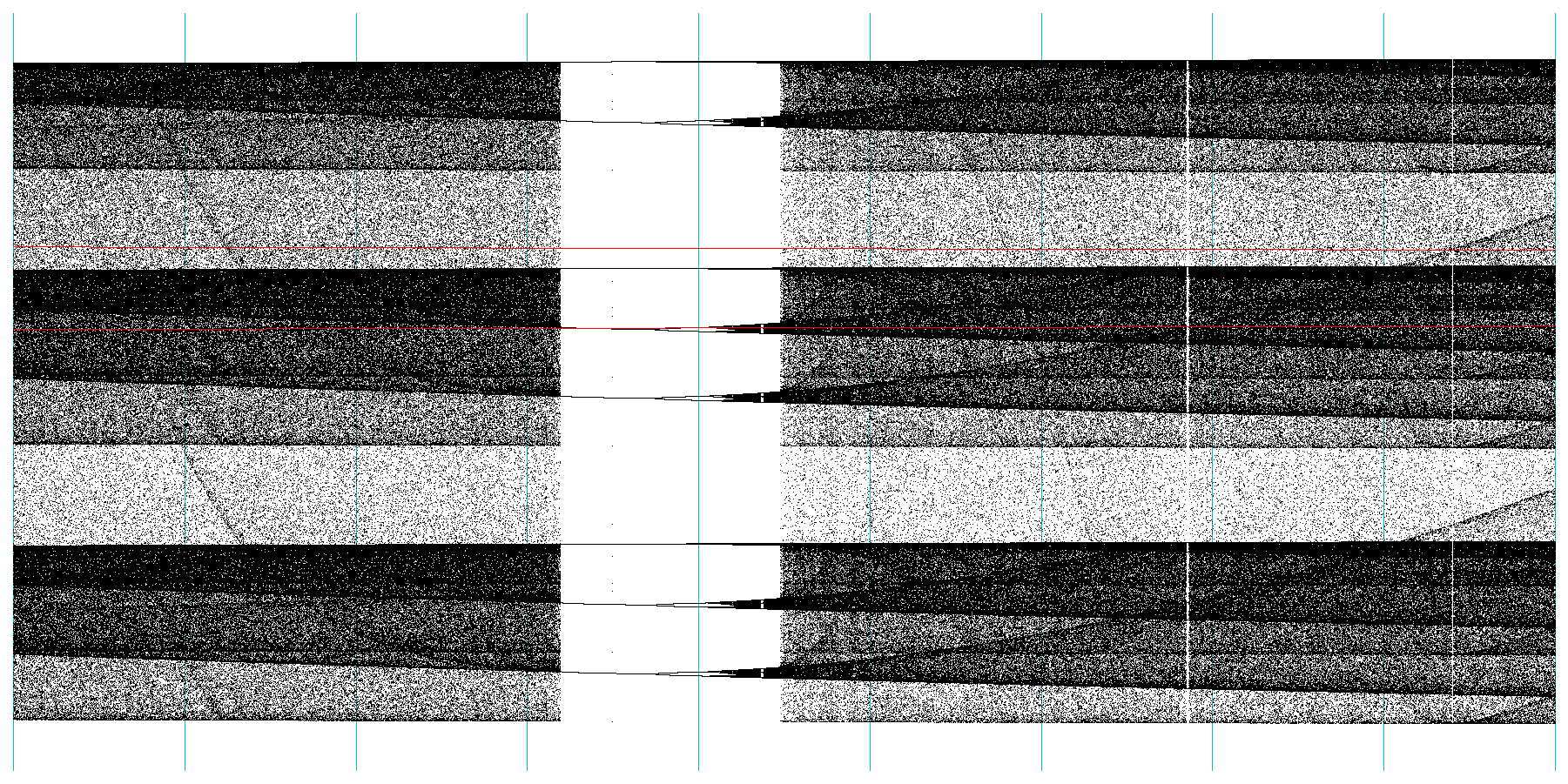}
%\caption{Here $b=4/11$.}\label{bbb4}
\end{center}
\end{subfigure}
\hfill
\begin{subfigure}{0.48\textwidth}
\begin{center}
\includegraphics[width=1.0\textwidth]{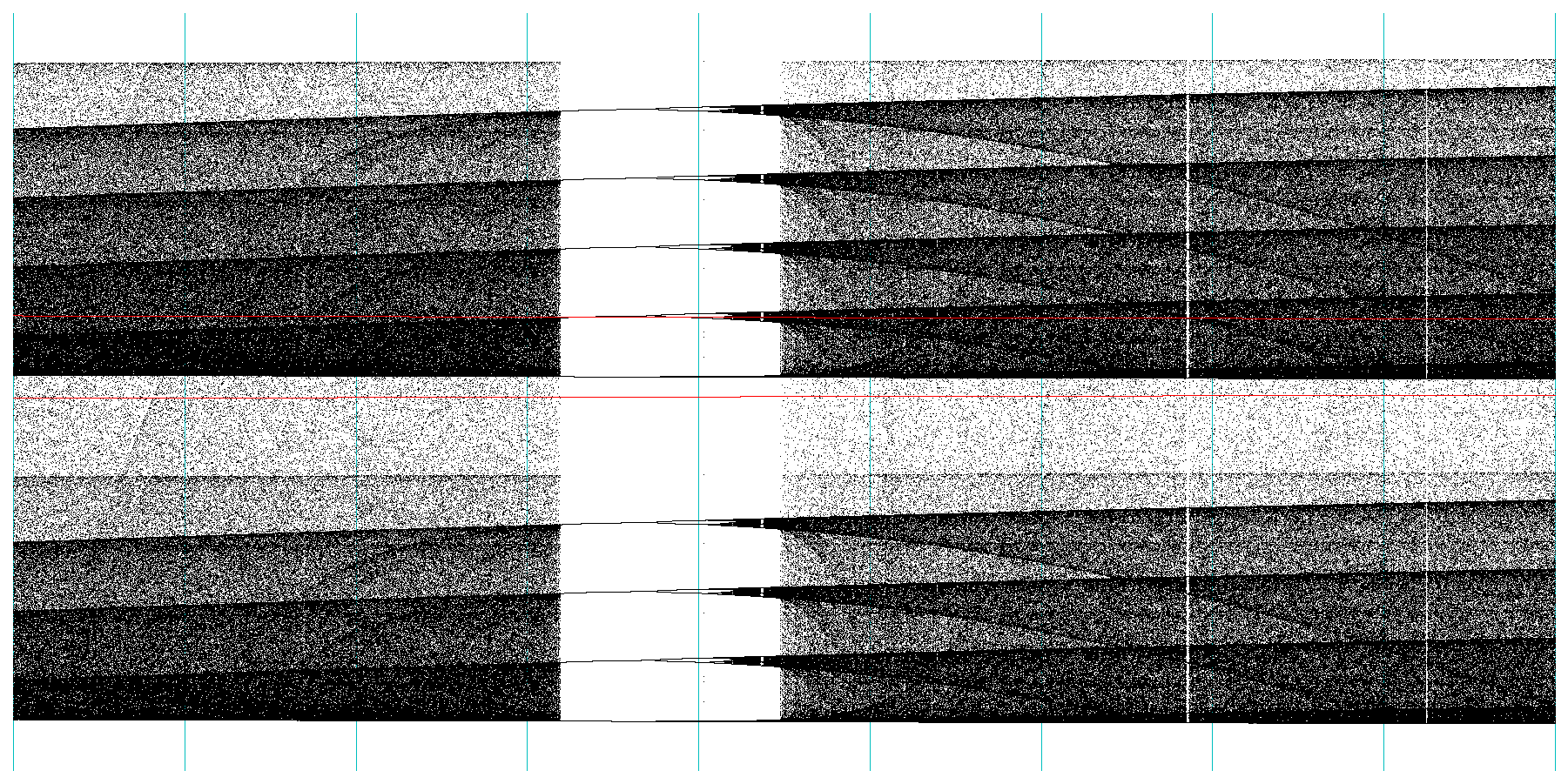}
%\caption{Here $b=5/11$.}\label{bbb5}
\end{center}
\end{subfigure}
\hfill
\begin{subfigure}{0.48\textwidth}
\begin{center}

\end{center}
\end{subfigure}
\caption{Bifurcation diagrams for
$b=k/11$, where $k=1,2,3,4,5$. On the horizontal axis there is $a$,
from $80$ to $89$; the vertical cyan lines are every $1$.
Horizontal red lines indicate the positions of the critical points.}\label{bbb}
\end{figure}

In order to understand this phenomenon, let us look at
Figure~\ref{5a3}. There $a=83.3$ is in the window we mention. Blue
lines indicate the position of the points of the attracting periodic
orbit of period 7.

Let us divide the interval $[b-1,b]$ into $n$ subintervals of length
$1/n$ each, and number those intervals $0,1,\dots,n-1$ from left to
right. For the large values of $a$, where there is an attracting
periodic orbit of period $n$, if we look where the consecutive (in
time) points of this orbit reside, we get the same result as for $G$,
that is, to get the location of the next point we just add $k$
modulo $n$. Since our window is the first large window to the left of
those large values of $a$, the change has to be the smallest possible
one. That is, at one moment instead of adding $k$, we add $k-1$ or
$k+1$ modulo $n$. In our example (Figure~\ref{5a3}) the point located
in the interval number 8 is mapped to the point in the interval number
1 (instead of number 0).

\begin{figure}[h!]
\begin{center}
\includegraphics[width=80truemm]{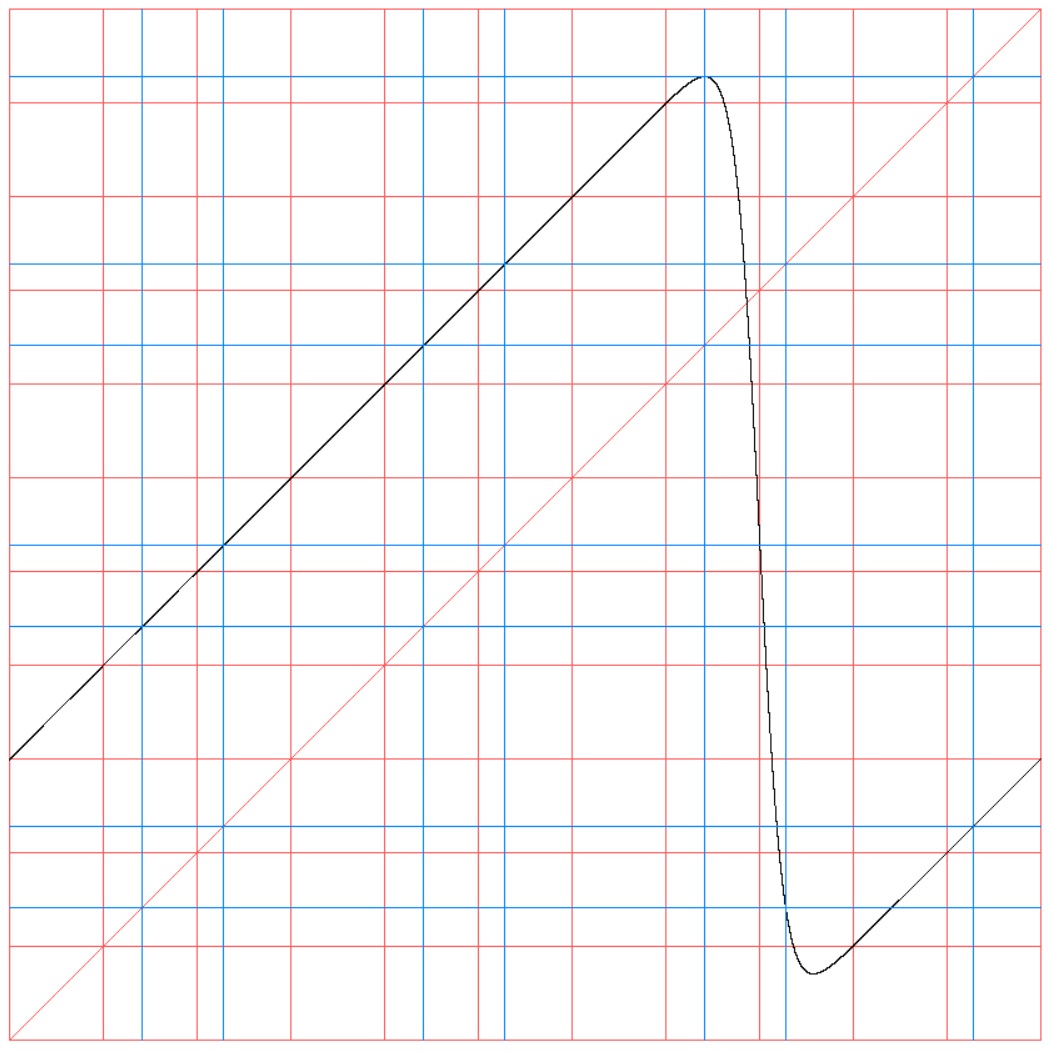}
\caption{Here $b=3/11$, $a=83.3$.}\label{5a3}
\end{center}
\end{figure}

This means that we get a periodic orbit of some period $m<n$ such that
$km$ modulo $n$ is 1 or $-1$. That is, for some $i$ we have
$km=in+1$ or $km=in-1$. Therefore $i/m$ is a Farey neighbor of $k/n$.
Since $m<n$, $i/m$ is one of the Farey parents of $k/n$.

As in the preceding explanations, we may replace $F$ by $\wf$, and
this explains why the relevant values of $a$ are almost independent of
$k$.

The only thing that is not quite clear is why we get the Farey parent
with the larger denominator. We speculate that a smaller change of the
period means a smaller change of $a$ (remember, that we are looking at
the big window closest to the infinite window with period $n$).

%As before, in order to convert what we wrote into a proof, we would
%need to make several subtle (and difficult) estimates.

Note that the explanations given in this section lack rigorous estimates,
so they cannot be considered rigorous proofs.

%\begin{figure}[h!]
%\begin{center}
%\includegraphics[width=80truemm]{wykres5a3.jpg}
%\caption{Here $b=3/11$, $a=83.3$.}\label{5a3}
%\end{center}
%\end{figure}

%\section{Conclusions}

%\vskip 0.5in
%{\bf CRediT authorship contribution statement.}
%
%\vspace{0.2cm}
%
%{\bf Jakub Bielawski:} Methodology, Investigation, Writing - Original Draft, Writing - Review \& Editing.
%{\bf Thiparat Chotibut:} Methodology, Investigation, Writing - Original Draft, Writing - Review \& Editing.
%{\bf Fryderyk Falniowski:} Methodology, Investigation, Writing - Original Draft, Writing - Review \& Editing.
%{\bf Michał Misiurewicz:} Methodology, Investigation, Writing - Original Draft, Writing - Review \& Editing.
%{\bf Georgios Piliouras: } Methodology, Investigation, Writing - Original Draft, Writing - Review \& Editing.

\vskip 0.3in
{\bf Declaration of competing interest.} The authors declare that they have no known competing financial interests or personal relationships that could have appeared to influence the work reported in this paper.
\vskip 0.3in
{\bf Data availability.} No data was used for the research described in the article.
\vskip 0.3in
{\bf Acknowledgments.} Jakub Bielawski and Fryderyk Falniowski acknowledge the support of the National Science Centre, Poland, grant number 2023/51/B/HS4/01343. Thiparat Chotibut acknowledges the funding support from the NSRF via the Program Management Unit for Human Resources \& Institutional Development, Research and Innovation [grant number B39G670016].

%%\bibliographystyle{abbrv}
%%\bibliography{ms}

\end{document}